\documentclass[a4paper,12pt]{article}

\usepackage[centertags]{amsmath}
\usepackage{amsfonts}
\usepackage{amssymb}
\usepackage{amsthm}
\usepackage{dsfont}
\usepackage{tikz, subfigure}
\usepackage{verbatim}
\usepackage{graphicx}

\newtheorem{theorem}{Theorem}[section]
\newtheorem{lemma}{Lemma}[section]
\newtheorem{cor}{Corollary}[section]
\newtheorem{prop}{Proposition}[section]

\title{The Dynamics of the Fibonacci Partition Function}
\author{Tom Kempton}
\begin{document}

\maketitle
\begin{abstract}
 In this article we study the local structure of the Fibonacci Partition Function by relating it to a cocycle over an irrational rotation. 
\end{abstract}
\section{Introduction}
In this article we are interested in partitions of natural numbers $n$ of the form
\[
n=x_1+\cdots +x_s
\]
where $x_1<x_2<\cdots <x_s$ are distinct Fibonacci numbers. We let $R(n)$ denote the number of such partitions of $n$ and can easily compute, for example, R(6)=2, since $6$ can be expressed as $5+1$ or as $3+2+1$. The sequence $(R(n))_{n=0}^{\infty}$ begins $1, 1, 1, 2, 1, 2, 2, 1, 3, 2, 2, 3, 1, 3$ and appears as sequence A000119 in the Online Encyclopedia of Integer Sequences (OEIS). This sequence has attracted a good deal of interest over the years, see for example \cite{Ardila, Berstel, Carlitz, ChowJones, EdsonZamboni, HoggattBasin,Klarner,Weinstein,Zhou}. Our interest in the sequence stems from an article of Alexander and Zagier \cite{AlexanderZagier}, where it appears in relation to Bernoulli convolutions. Recently, Chow and Slattery \cite{ChowSlattery} gave an exact formula for $R(n)$, whose behaviour they describe as erratic.

In this article, we show that $R(n)$ can be expressed in terms of ergodic sums over an irrational rotation. This result readily yields results about the local multiplicative structure of $R$ and explains and quantifies the erratic behaviour. Chow and Slattery also proved results about the growth of $A(H):=\sum_{n=0}^H R(n)$, we prove further results in this direction in Section \ref{SlatterySection}.

Let $T:\left[\dfrac{-1}{\varphi^2},\dfrac{1}{\varphi}\right)\to \left[\dfrac{-1}{\varphi^2},\dfrac{1}{\varphi}\right)$ be given by \[T(y)=\left\lbrace\begin{array}{cc}y+\dfrac{1}{\varphi^2}&y\in\left[\dfrac{-1}{\varphi^2},\dfrac{1}{\varphi^3}\right]\\
y+\dfrac{1}{\varphi^2}-1&y\in\left[\dfrac{1}{\varphi^3},\dfrac{1}{\varphi}\right]
\end{array}\right. .\] The function $T$ is a rotation by angle $\frac{1}{\varphi^2}$ where $\varphi=\frac{1+\sqrt{5}}{2}$ is the golden mean. Usually irrational rotations are defined on $[0,1)$, we have chosen a different interval of length $1$ for reasons which will make sense later. 

Let $h:\left(\dfrac{-1}{\varphi^2},\dfrac{1}{\varphi^3}\right)\cup\left(\dfrac{1}{\varphi^3},\dfrac{1}{\varphi}\right)\to \mathbb R$ be the unique continuous function satisfying
\[
h(y)=\left\lbrace \begin{array}{cc}
1+h(-\varphi y)&y \in \left(\dfrac{-1}{\varphi^2},\dfrac{-1}{\varphi^4}\right]\\
1&y\in\left[\dfrac{-1}{\varphi^4},0\right]\\
\dfrac{h(-\varphi y+\frac{1}{\varphi})}{1+h(-\varphi y+\frac{1}{\varphi})}&y\in \left[0,\dfrac{1}{\varphi^3}\right)\\
h(-\varphi y+\dfrac{1}{\varphi})&y\in\left(\dfrac{1}{\varphi^3},\dfrac{1}{\varphi}\right)
\end{array} \right.
\]

The function $h$, pictured in Figure \ref{Fig1}, is a `Devil's staircase', taking constant values on a set of intervals of full Lebesgue measure while being continuous. The following is our main theorem. 
\begin{theorem}\label{Thm1}
We have \[
R(n)=\exp\left(\sum_{k=0}^{n-1} \log(h(T^k(0)))\right)
\]
for all $n\geq 1$.
\end{theorem}

\begin{figure}\label{Fig1}
    \centering
\includegraphics[scale=0.5]{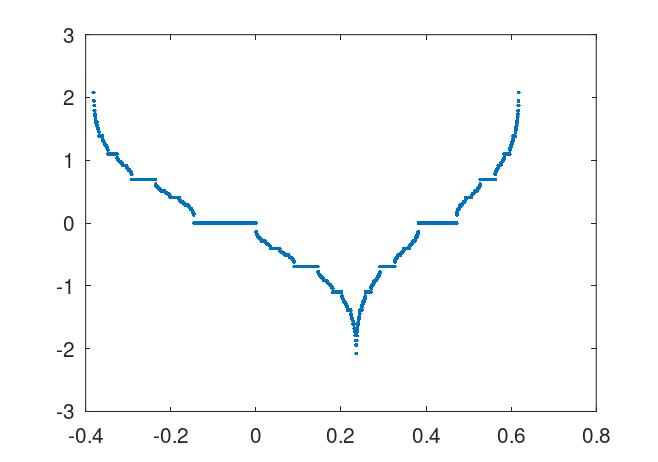}
    \caption{The function $\log(h)$, drawn by plotting $y_n=T^n(0)$ against $\log\left(\frac{R_n+1}{R_n}\right)$ for $0\leq n\leq 2582$. Note that $\log(h)\to\pm\infty$ on the boundary of its domain.}
    \label{fig:enter-label}
\end{figure}

Given a vector $P=(p_1,\cdots p_k)$ where $k\in\mathbb N$ and each $p_i\in\mathbb Q$, we say $R$ contains patch $P$ at time $n$ if $R(n+i)=p_iR(n)$ for each $i\in\{1,\cdots,k\}$. For example, if $P=(1,1)$ then $R$ contains patch $P$ at time $n$ whenever $R(n)=R(n+1)=R(n+2)$.

\begin{theorem}\label{Thm2}
For any patch $P$ the set 
\[
\{n\in\mathbb N: R \mbox{ contains patch $P$ at time $n$} \}
\]
is a cut and project set. 
\end{theorem}

These two theorems make it very easy to read off statements about the local multiplicative structure of $R$. To illustrate the results, we give the following two corollaries which are more or less immediate.

\begin{cor}\label{Cor1}
The asymptotic density of the set of $n$ for which $R(n)=R(n+1)$ is equal to $\frac{1}{\varphi^4}+\frac{1}{\varphi^5}=\frac{1}{\varphi^3}\approx 0.236.$
\end{cor}
\begin{proof}
We see that $R(n)=R(n+1)$ whenever $h(T^n(0))=1$, which happens whenever $T^n(0)\in\left[\dfrac{-1}{\varphi^4},0\right]\cup \left[\dfrac{1}{\varphi^2},\dfrac{1}{\varphi^2}+\dfrac{1}{\varphi^5}\right]$. The irrational rotation $T$ is uniquely ergodic and so every orbit equidistributes with respect to Lebesgue measure, entering the region above with frequency equal to the sum of their lengths, i.e. $\frac{1}{\varphi^4}+\frac{1}{\varphi^5},$ as required.
\end{proof}
Theorem \ref{Thm2} really gives finer information than Corollary \ref{Cor1} on the structure set of $n$ for which $R(n)=R(n+1)$, for example one could express the set in terms of substitution sequences, or write down a speed of convergence for the statement in Corollary \ref{Cor1}.
\begin{cor}\label{Cor2}
The largest $k\in\mathbb N$ for which there exist $n\in\mathbb{N}$ such that the sequence $R(n), R(n+1),\cdots R(n+k)$ is increasing is $k=3$. This occurs uniquely at $n=0$, there do not exist four other consecutive terms of the sequence upon which $R$ is increasing.
\end{cor}

\begin{proof}
The function $h$ is non negative on $[\frac{-1}{\varphi^2},0]\cup [\frac{1}{\varphi^2},\frac{1}{\varphi}].$ It is strictly negative on the complement of this set. Our irrational rotation can only jump over the negative interval from the point $0$. So the longest orbit segment which remains in the increasing region is $0, \frac{1}{\varphi^2},\frac{2}{\varphi^2}-1.$ This corresponds to terms $0,1,2,3$ of the sequence.
\end{proof}

%\begin{cor}\label{cor3}
%For $k\in\mathbb N$, large jumps $\frac{R(n+1)}{R(n)}\geq k$ occur with frequency ....
%\end{cor}

Our method involves first converting the problem to one involving sums of powers of the golden mean. Once that has been done, we use techniques developed in our work with Batsis \cite{BK1,BK2} to prove Theorem \ref{Thm1}. Analysis of the matrices involved gives the devil's staircase structure of $h$, which was a surprise to us, and which leads easily to Theorem \ref{Thm2} and its corollaries.

\section{A Dynamical Approach to $R$}

The $n$th Fibonacci number can be written
\begin{equation}\label{FnForm}
F_n=\dfrac{\varphi^n-\psi^n}{\sqrt{5}}
\end{equation}
where $\psi=\dfrac{1-\sqrt{5}}{2}=\dfrac{-1}{\varphi}$. The function $R(n)$ counts ways to write $n$ as the sum of distinct Fibonacci numbers, i.e. as a sum of the form
\[
n=\sum_{i=1}^k a_iF_{k+2-i}
\]
where each $a_i\in\{0,1\}$. This contains terms from the set $\{F_i:i\geq 2\}$, since $F_1=F_2=1$. 

It is useful to us to separate out the two terms in (\ref{FnForm}), not least because sums of powers of Pisot numbers (such as the golden mean) have been extensively studied. We begin with a lemma which lets us separate out these terms.

\begin{lemma}\label{xnDef}
Let $n\in\mathbb N$. For any $a_1\cdots a_k\in\{0,1\}^k$ with \begin{equation}\label{xndef}\sum_{i=1}^k a_iF_{k+2-i}=n,\end{equation} we let
\[
x_n=\sum_{i=1}^k a_i\varphi^{k+2-i}
\]
and
\[
y_n=\sum_{i=1}^k a_i \psi^{k+2-i}.
\]
Then $x_n$ and $y_n$ are well defined functions of $n$ which are independent of the choice of $a_1 \cdots a_k$ satisfying (\ref{xndef}). Furthermore, the partition function $R$ which we study also satisfies
\begin{align*}
R(n)&=\#\{a_1\cdots a_k\in\{0,1\}^k:k\in\mathbb N, a_1=1, x_n=\sum_{i=1}^k a_i\varphi^{k+2-i}\}\\
&=\#\{a_1\cdots a_k\in\{0,1\}^k:k\in\mathbb N, a_1=1, y_n=\sum_{i=1}^k a_i\psi^{k+2-i}\}
\end{align*}
\end{lemma}

\begin{proof}
It suffices to show that if $a_1\cdots a_k$ and $b_1\cdots b_j$ are two finite words made of $0$s and $1$s then
\begin{align}
\sum_{i=1}^k a_iF_{k+2-i}&=\sum_{i=1}^jb_iF_{j+2-i}\label{0eq}\\
\iff \sum_{i=1}^k a_i\varphi^{k+2-i}
&=\sum_{j=1}^m b_i\varphi^{j+2-i}\label{1steq}\\
\iff \sum_{i=1}^k a_i \psi^{k+2-i}&=\sum_{i=1}^j b_i \psi^{j+2-i}.
\end{align}

The final equivalence here holds because $\varphi$ and $\psi$ have the same minimal polynomial. The fact that (\ref{0eq}) implies (\ref{1steq}) is proved in part 3 of Proposition \ref{StructureProp}\footnote{ It's worth noting that analogous statements don't always hold if we replace $F_n$ with other recurrence sequences whose growth is governed by a Pisot number.}.\end{proof}

We define 
\[
\overline X=\left\{(x_n,y_n):n\in\mathbb N\right\}.
\]
The set $\overline X$ is well studied, it forms a strip through a lattice, see Figure \ref{Fig2}. 

We let maps $T_0, T_1, T_{-1}, S_0, S_1:\mathbb R\to\mathbb R$ be given by \[T_i(x)=\varphi x+i\varphi^2\] and \[S_i(y)=\psi y+i\psi^2.\]

These maps satisfy that
\begin{equation}\label{TStructure}
    \sum_{i=1}^k (a_i\varphi^{k+2-i},a_i\psi^{k+2-i})=(T_{a_1}\circ \cdots \circ T_{a_k}(0),S_{a_1}\circ \cdots \circ S_{a_k}(0)).
\end{equation}
and so
\begin{equation}\label{XStructure}
\overline X=\{(T_{a_1}\circ \cdots \circ T_{a_k}(0),S_{a_1}\circ \cdots \circ S_{a_k}(0)):k\geq 1, a_i\in\{0,1\}\}.
\end{equation}

The next proposition shows the relationship between $\overline X$ and $\mathbb N$ as well as describing nearest neighbour dynamics on $\overline X$ which will later allow us to compute $R(n)$ easily.

%and $X=\pi(\overline X)$ where $\pi(x,y)=x$. That is, $\overline X$ is the set of sums of distinct elements of $(\overline F_n)$ and $X$ is the projection of $\overline X$ onto the first coordinate. By analogy with the sequence $R(n)$, we let $\overline R(x,y)$ be the number of ways of representing $(x,y)\in \overline X$ as sums of distinct elements of $\overline F_i$. Finally we order elements of $\overline X$ by their first coordinate and order the integers in the usual way.

\begin{prop}\label{StructureProp} The following hold:
\begin{enumerate}
\item The set $\overline X$ can be expressed \begin{equation}\label{CPStructure}\overline{X}=\left\{\left(n+m\varphi, n+m\psi\right):~n,m\in\mathbb Z,~n+m\psi\in\left[\dfrac{-1}{\varphi^2},\frac{1}{\varphi}\right]\right\}.\end{equation}
\item The successor map $s:\overline{X}\to\overline{X}$ which maps $(x_n,y_n)$ to $(x_{n+1},y_{n+1})$ acts by
\[
s(x,y)=\left\lbrace\begin{array}{cc}(x+1+\varphi,y+1+\psi)& y \in \left[\dfrac{-1}{\varphi^2},\dfrac{1}{\varphi^3}\right) \\ 
(x+\varphi,y+\psi) & y\in \left[\dfrac{1}{\varphi^3},\dfrac{1}{\varphi}\right]\end{array}\right. .
\]
\item The map $g:\overline X\to\mathbb N\cup\{0\}$, $g(x,y)=\dfrac{x-y}{\sqrt{5}}$ is an order preserving bijection mapping $(x_n,y_n)$ to $n$.
%\item We can write
%\[
%n=a_2F_{n}+ a_3 F_{n-1}+\cdots a_n F_2
%\]
%for $a_2\cdots a_n\in\{0,1\}^n$ if and only if
%\[
%g^{-1}(n)=T_{a_2}\circ\cdots T_{a_n}(0,0).
%\]
\end{enumerate}
\end{prop}
\begin{proof} Parts 1 and 2 are standard results in arithmetic dynamics, indeed the generalisation in which $\varphi$ is replaced with the largest root of $x^3=x^2+x+1$ gave rise to the Rauzy fractal \cite{Rauzy}, which in turn sparked a great deal of study relating dynamics and the arithmetic of algebraic integers. 

To prove part 1, we first note that, by repeatedly applying \[(\varphi^{n+2},\psi^{n+2})=(\varphi^{n+1},\psi^{n+1})+(\varphi^n,\psi^n)\] to the largest terms in the polynomials defining $(x_n,y_n)$, we can express any point $(x,y)$ in $\overline X$ as
\[
(x,y)=(a+b\varphi,a+b\psi)
\]
for some choices of $a,b\in\mathbb Z$. The maps $S_0$ and $S_1$ which generate the points $y_n$ are contractions, and in particular no points $S_{a_1}\circ\cdots S_{a_k}(0)$ can escape the interval $[-\psi^2,-\psi]=[-\frac{1}{\varphi^2},\frac{1}{\varphi}]$ (which is the attractor of the iterated function system $\{S_0, S_1\}$). Thus the $y$ coordinates of points in $\overline X$ are contained in this interval and we see 
\[
\overline X\subset \left\{\left(a+b\varphi, a+b\psi\right):~a,b\in\mathbb Z,~a+b\psi\in\left[\dfrac{-1}{\varphi^2},\frac{1}{\varphi}\right]\right\}
\]

Proving the reverse inclusion is a short argument which can be found, for example, in section 3.1 of \cite{BK1}. Essentially one shows that for any point $(x,y)$ in the set on the right hand side of equation \ref{CPStructure} one can map $(x,y)$ to either $(T_0^{-1}(x), S_0^{-1}(y))$ or $(T_1^{-1}(x), S_1^{-1}(y))$ without leaving the set, and by doing this enough times one ends up in a finite set around $(0,0)$. It then remains to check by hand that each of these finitely many points can be reached from $(0,0)$ by applying maps $(T_i,S_i)$, and so they are contained in $\overline X$.

Part 2, which is again proved carefully in Section 3.1 of \cite{BK1}, follows directly from seeing that the set $\overline X-\overline X$ of differences between pairs of points in $\overline X$ has a discrete structure similar to that of $\overline X$ and analysing the smallest positive elements of this set.

To check part 3 we just check that $g(0,0)=0$ and $g(s(x,y))=g(x,y)+1$ for all $(x,y)\in\overline X$, which is a direct consequence of the form taken by the successor map $s$.

%Part 4 is a direct consequence of equation \ref{TStructure} and the fact that $g(\overline F_n)=F_n$ for $n\in\mathbb N$.

\end{proof}
\begin{figure}\label{Fig2}
    \centering
    \includegraphics[scale=0.6]{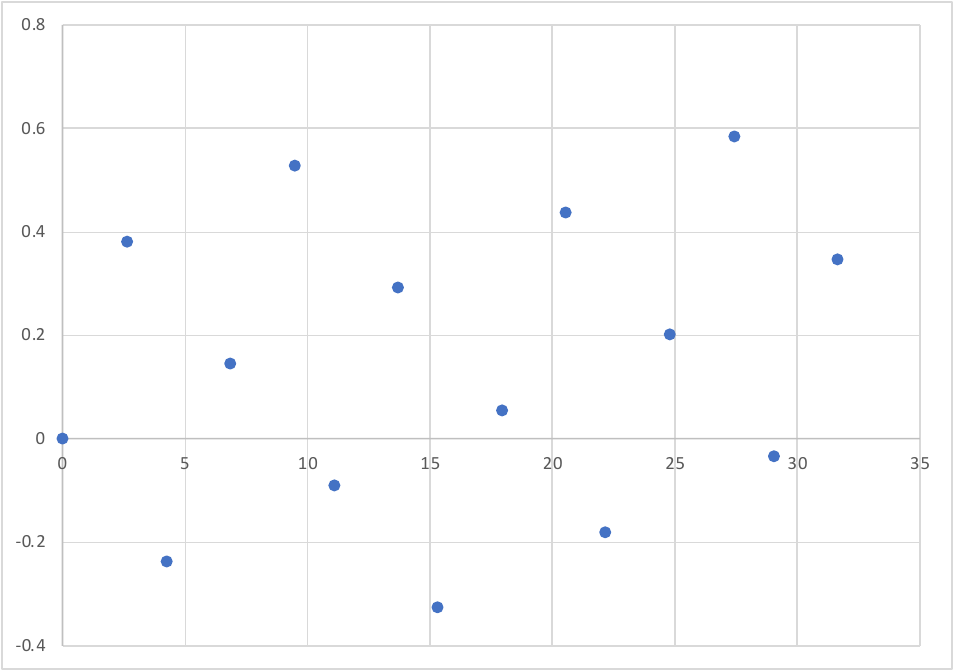}
    \caption{The first part of the set $\overline X=\{(x_n,y_n):n\in\mathbb N\}$}
    \label{fig:enter-label}
\end{figure}

The Zeckendorf representation \cite{Zeckendorf} of $n\in\mathbb N$ is the unique representation $n=F_{a_1}+F_{a_2}+\cdots F_{a_k}$ where each $a_i\geq 2$ and $a_{i+1}\geq a_i+2$. That is, the Zeckendorf representation gives the unique partition of $n$ into non-consecutive Fibonacci numbers $F_n$. This is analagous to the greedy $\varphi$-expansion of $x$. We write $<n>=b_1...b_k$ where $b_i=1$ if $F_{k+2-i}$ is in the Zeckendorf representation, $b_i=0$ otherwise. In this notation, each integer can be written in the form
$<n>=10^{d_1}10^{d_2}\cdots 10^{d_r}$ with $d_r\geq 0$ and all other $d_i\geq 1$.

We now show how information about $R(n)$ and $R(n-1)$ can be read off from a product of matrices involving the Zeckendorf representation. Our result is very much in the spirit of Berstel \cite{Berstel}, but our coding is different in order to allow information on both $R(n)$ and $R(n-1)$ in the same matrix product. 

Let \[
A_1=\left(\begin{array}{ccc}1&0&1\\0&0&1\\0&0&0\end{array}\right),~A_0=\left(\begin{array}{ccc}1&0&0\\1&0&1\\0&1&0\end{array}\right).
\]

\begin{prop}\label{MatrixProp}
Suppose $<n>=b_1\cdots b_k$. Then
\[
R(n)=(1~0~0)A_{b_1}\cdots A_{b_k} \left(\begin{array}{c} 1\\0\\0\end{array}\right)
\]
and
\[
R(n-1)=(1~0~0)A_{b_1}\cdots A_{b_k} \left(\begin{array}{c} 0\\1\\1\end{array}\right).
\]
\end{prop}
It is important here that we use the Zeckendorf coding, and not another partition of $n$ by Fibonacci numbers.
\begin{proof}
Let $<n>=b_1\cdots b_k$ and suppose that we have another partition $n=\sum_{i=1}^k a_iF_{k+2-i}$ with each $a_i\in\{0,1\}$. Then by Lemma \ref{xnDef}
\[
x_n=\sum_{i=1}^k b_i\varphi^{k+2-i}=\sum_{i=1}^k a_i\varphi^{k+2-i},
\]
and so
\begin{equation}\label{ZeroPoly}
\sum_{i=1}^k (b_i-a_i)\varphi^{k+2-i}=0.
\end{equation}
Writing $c_i=b_i-a_i\in\{-1,0,1\}$, we see that (\ref{ZeroPoly}) holds if and only if $T_{c_k}\circ \cdots T_{c_1}(0)=0$, recalling that $T_i(x)=\varphi x +i\varphi^2$. Thus, to compute $R(n)$ where $<n>=b_1\cdots b_k$ we need to count the number of $a_1\cdots a_k\in\{0,1\}^k$ with $T_{c_k}\circ \cdots T_{c_1}(0)=(0)$. The $c_i=b_i-a_i$ are allowed to take values $b_i, b_i-1$, since $a_i\in\{0,1\}$. In particular, if $b_i=1$ then $c_i$ must be in $\{0,1\}$, if $b_i=0$ then $c_i$ must be in $\{-1,0\}$. 

Lemma 3.1 of \cite{AFKP} with $\varphi$ gives, after a short calculation, that for any such $c_1\cdots c_k$ and for any $j\leq k$,
\begin{equation}\label{diffs}
T_{c_j}\circ \cdots T_{c_1}(0)=(0)\in\{0,\varphi,\varphi^2\}.
\end{equation}

Let $v_1=0$, $v_2=\varphi$, $v_3=\varphi^2$. Then we see that the matrices $A_1, A_0$ above satisfy
\[
A_1(i,j)=\left\lbrace\begin{array}{cc} 1 & v_j\in\{T_1(v_i),T_0(v_i)\}\\
0 & \mbox{ otherwise }\end{array}\right. 
\]
and
\[
A_0(i,j)=\left\lbrace\begin{array}{cc} 1 & v_j\in\{T_0(v_i),T_{-1}(v_i)\}\\
0 & \mbox{ otherwise }\end{array}\right. .
\]
Then given a Zeckendorf representation $b_1\cdots b_k$, $(A_{b_1}\cdots A_{b_k})(1,1)$ counts exactly the number of $a_1\cdots a_k\in\{0,1\}^k$ with $\sum_{i=1}^k(b_i-a_i)\varphi^{k+2-i}=0$, which is exactly $R(n)$.

The two values $(A_{b_1}\cdots A_{b_k})(1,2)$ and $(A_{b_1}\cdots A_{b_k})(1,3)$ count the number of $a_1\cdots a_k\in\{0,1\}^k$ for which $\sum_{i=1}^k(b_i-a_i)\varphi^{k+2-i}$ equals $\varphi$ or $\varphi^2$ respectively. One of these numbers will be zero, and one will be $R(n-1)$, depending on whether $x_n-x_{n-1}$ is equal to $\varphi$ or equal to $\varphi^2=1+\varphi$. We can see that exactly one of these situation holds by looking at the form of the successor map. Thus we see that the proposition holds.

We are using here that $b_1\cdots b_k$ is the Zeckendorf representation, if it was another representation then it would also be possible for the sum in (\ref{diffs}) to take values $-1,-\varphi$ and so we would need to use $5\times 5$ matrices instead of our $3\times 3$ matrices $A_0$ and $A_1$.

%This can be proved fairly directly in half a page. We omit the proof, since the fact that such a pair of finite matrices $A_0$ and $A_1$ exist appears as Lemma 3.1 of \cite{AFKP}, computing the precise form of the matrices is a short exercise. 
\end{proof}

\section{Proof of Main Theorems}\label{Section3}
Theorem \ref{Thm1} is proved by understanding the projective action of the matrices $A_0, A_1$.

We define $\pi:\{(a~b~c)\in \mathbb R^3:a\neq 0\}\to\mathbb R^2$ by 
\[
\pi(a~b~c)=\left(\frac{b}{a},\frac{c}{a}\right).
\]
Define maps $e_0,e_1:\mathbb R^2\to\mathbb R^2$ by
\[
e_0(x,y)=\left(\frac{y}{1+x},\frac{x}{1+x}\right),~e_1(x,y)=\left(0,1+x\right)
\]
Then we see by inspection that
\[
\pi((a~b~c)A_i)=e_i(\pi(a~b~c))
\]
for $i\in\{0,1\}$. 

Suppose that $<n>=b_1\cdots b_k$, $<m>=b_1\cdots b_kb_{k+1}$ and
\[
(1~0~0)A_{b_1}\cdots A_{b_k}=(a~b~c).
\]

We have seen before that $y_m=S_{b_{k+1}}(y_n)=\frac{-y}{\varphi}+b_{k+1}\frac{1}{\varphi^2}$. 

We also see from our matrices $A_i$ that exactly one of $b$ and $c$ will be non-zero. It follows from the definition of the successor map in Proposition 2.1 part 2 that $b>0$ if and only if $y_n\in\left[\dfrac{-1}{\varphi^2},0\right)$, $c>0$ if and only if $y_n\in\left(0,\dfrac{1}{\varphi}\right]$. 

Before working out $h$, we compute $k(y_n)=\frac{R(n)}{R(n-1)}$ since it is a slightly easier computation. We recover formulas for $h$ using $h(y)=k(T(y))$.

{\bf Case 1:} Suppose that $y_n\in\left[\dfrac{-1}{\varphi^2},0\right]$ and $b_{k+1}=0$. Then $c=0$, \[(a~b~c)A_0=(a~b~0)A_0=(a+b~0~b).\] In this case $R(n)=a$ and $b=R(n-1)$ giving $a=k(y_n)b$. So we have computed $R(m)=a+b$ and $R(m-1)=b$. Then
\[
k\left(\frac{-y_n}{\varphi}\right)=\frac{a+b}{b}=1+k(y_n). 
\]

{\bf Case 2:} Suppose that $y_n\in[0,\frac{1}{\varphi}]$ and $b_{k+1}=0$. Then $b=0$ and
\[(a~b~c)A_0=(R(n)~0~R(n-1))A_0=(R(n)~R(n-1)~0).\]
Then $k\left(\frac{-y_n}{\varphi}\right)=\frac{R(n)}{R(n-1)}=k(y_n).$

{\bf Case 3:}
Suppose that $y_n\in\left[\frac{-1}{\varphi^2},0\right]$ and $b_{k+1}=1$. Then $c=0$ and
\[(a~b~c)A_1=(R(n)~R(n-1)~0)A_1=(R(n)~0~R(n)+R(n-1)).\]
Thus
\[k\left(\frac{-y_n}{\varphi}+\frac{1}{\varphi^2}\right)=\frac{R(n)}{R(n)+R(n-1)}=\frac{k(y_n)}{1+k(y_n)}.
\]

{\bf Case 4:}
Suppose that $y_n\in [0,\frac{1}{\varphi^3}]$ and $b_{k+1}=1$. Then $b=0$ and
\[(a~b~c)A_1=(R(n)~0~R(n-1))A_1=(R(n)~0~R(n)).\]
Then
\[
k\left(\frac{-y_n}{\varphi}+\frac{1}{\varphi^2}\right)=\frac{R(n)}{R(n)}=1.
\]
Note that for case 4, we dealt with $y_n\in [0,\frac{1}{\varphi^3}]$ rather than $[0,\frac{1}{\varphi}]$, this is because we $b_{k+1}$ can only be $1$ if $b_k$ was zero, which puts restrictions on $y_k$. Collecting the four identities above and manipulating them to an easier form gives
\[
k(y)=\left\lbrace \begin{array}{cc}k(-\varphi y)&y\in\left[\dfrac{-1}{\varphi^2},0\right]\\
1+k(-\varphi y)&y\in \left[0,\dfrac{1}{\varphi^3}\right)\\
1&y\in\left[\dfrac{1}{\varphi^3},\dfrac{1}{\varphi^2}\right]\\
\dfrac{k(-\varphi y+\frac{1}{\varphi})}{1+k(-\varphi y+\frac{1}{\varphi})}&y\in\left[\dfrac{1}{\varphi^2},\dfrac{1}{\varphi}\right]\end{array}\right.
\]
Finally, noting that $h(y)=k(T(y))$, we recover form of $h$ given in the original definition. 

The proof of Theorem \ref{Thm1} is completed by noting that 
\[
R(n)=R(0)\prod_{k=0}^{n-1}\dfrac{R(T^{k+1}(0))}{R(T^k(0))}=1\times \prod_{k=0}^{n-1}h(T^k(0))=\exp\left(\sum_{k=0}^{n-1}\log(h(T^k(0)))\right). 
\]

We now turn to the proof of Theorem \ref{Thm2}. Given a patch $P=(p_1,\cdots, p_k)$, we see, by rewriting the definition of the patch, that $R$ contains patch $P$ at time $n$ if and only if $R(n+1)=p_1R(n)$ and $R(n+i)=\frac{p_i}{p_{i-1}}R(n+i-1)$ for $2\leq i \leq k$. Consider the region \[W_{p_1,\cdots,p_k}=\left\{y\in \left(\dfrac{-1}{\varphi^2},\dfrac{1}{\varphi}\right):h(y)=p_1 \mbox{ and }h(T^i(y))=\frac{p_i}{p_{i-1}} \mbox{ for }2\leq i\leq k\right\}
\]
This will be a (possibly empty) union of intervals, since the region upon which $h$ takes some given value is always empty or the union of two intervals.

Then, using again Theorem \ref{Thm1}, we see that $R$ has patch $P$ at time $n$ if and only if $y_n=T^n(0)\in W_{p_1,\cdots,p_k}$. Now using the results of Proposition \ref{StructureProp} gives that $R$ has patch $P$ at time $n$ if and only if
\[
n\in\left\{\frac{n+m\varphi-(n+m\psi)}{\sqrt{5}}:(n,m)\in\mathbb Z^2, n+m\psi\in W_{p_1,\cdots,p_k}\right\}.
\]
This set is a cut and project set, and so the proof of Theorem \ref{Thm2} is complete.

\section{Growth of the function $A(H)$}\label{SlatterySection}

Let $A(H)=\sum_{k=0}^H R(k)$. Chow and Slattery proved that  
\[
\liminf_{H\to\infty} \dfrac{A(H)}{H^{\frac{\log 2}{\log \varphi}}}=c_1
\]
and
\[
\limsup_{H\to\infty} \dfrac{A(H)}{H^{\frac{\log 2}{\log \varphi}}}=c_2
\]
where $c_1\approx 0.52534 \cdots$ and $c_2\approx 0.54338\cdots$. They also plotted a graph of $\dfrac{A(H)}{H^{\frac{\log 2}{\log \varphi}}}$, whose behaviour seems asymptotically log-periodic. We have replotted this graph in Figure 3 and prove this asymptotic log-periodicity.

\begin{theorem}
The function $H\to \dfrac{A(H)}{H^{\frac{\log 2}{\log \varphi}}}$ is asymptotically log-periodic with a limit given in terms of the cumulative distribution function of the Bernoulli convolution associated to $\varphi$.
\end{theorem}

\begin{figure}\label{AFigure}
    \centering
\includegraphics[scale=0.32]{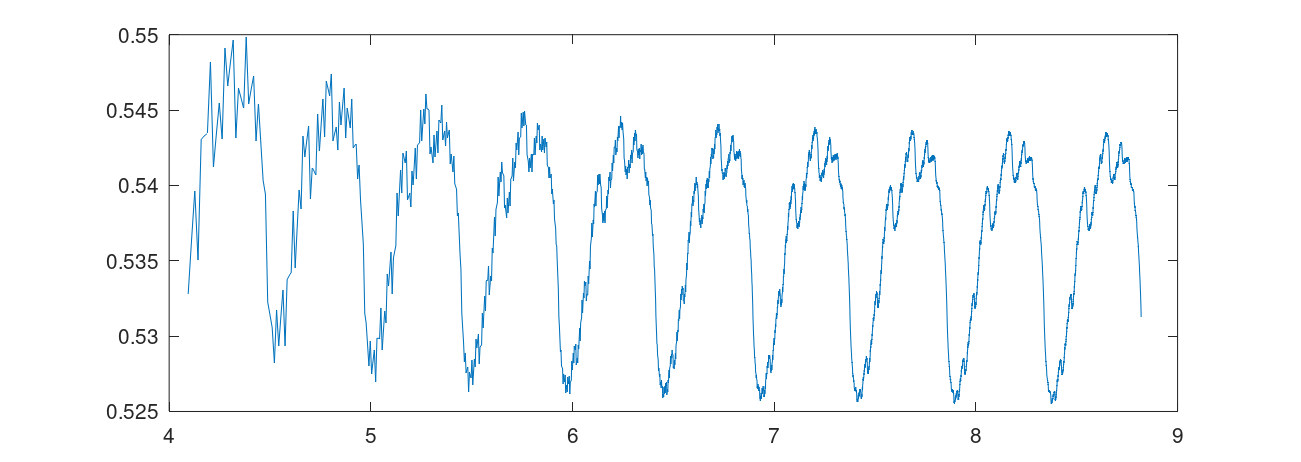}
    \caption{A plot of $\log(H)$ against $\dfrac{A(H)}{H^{\frac{\log 2}{\log \varphi}}}$ for $60\leq H\leq 6765$.}
    \label{fig:enter-label}
\end{figure}

In the following we use $f(k)\sim g(k)$ to mean $\lim_{k\to\infty}\frac{f(k)}{g(k)}=1$.

Let the Bernoulli convolution $\nu_{\varphi}$ be probability measure on $\mathbb R$, supported on $[0,\varphi]$, with  \[\mu(A)=\mathbb P\left(\sum_{i=1}^{\infty}a_i\varphi^{-i}\in A\right)\] where each $a_i\in\{0,1\}$ is picked independently with probability $(1/2,1/2)$. Let $G_{\varphi}$ be the cumulative distribution function given by $G_{\varphi}(x)=\nu_{\varphi}[0,x]$. The Bernoulli convolution $\nu_{\varphi}$ is a continuous (but not absolutely continuous) measure, and so $G_{\varphi}:\mathbb R\to\mathbb R$ is a continuous but not differentiable function. 

Given a fixed prefix $a_1\cdots a_k$, we have \[\left\{\sum_{i=1}^{\infty} a_i\varphi^{-i}:a_{k+1}\cdots \in\{0,1\}^{\mathbb N}\right\}=\left[\sum_{i=1}^k a_i\varphi^{-i},\sum_{i=1}^k a_i\varphi^{-i} +\varphi^{1-k}\right].\]
This gives that for any $x_H$ and any $k\in\mathbb N$
\begin{equation}\label{CumulativeBound}
G_{\varphi}\left(\frac{x_H}{\varphi^{k+2}}\right)\leq\frac{1}{2^k}\#\{a_1\cdots a_{k}\in\{0,1\}^k:\sum_{i=1}^k a_i\varphi^{-i}\leq \frac{x_H}{\varphi^{k+2}}\}\leq G_{\varphi}\left(\frac{x_H+\varphi^3}{\varphi^{k+2}}\right).
\end{equation}
Now let $F_{k-1}\leq H< F_k$ and set $\gamma=\frac{H}{F_{k-1}}$. Then for all $n\leq H<F_k$, $R(n)=\#\{a_1\cdots a_{k}\in\{0,1\}^k:\sum_{i=1}^k a_i\varphi^{k+2-i}=x_n\}$.

Thus the central term in (\ref{CumulativeBound}) satisfies
\begin{align*}
    & \frac{1}{2^k}\#\{a_1\cdots a_{k}\in\{0,1\}^k:\sum_{i=1}^k a_i\varphi^{-i}\leq \frac{x_H}{\varphi^{k+2}}\}\\&=\frac{1}{2^k}\#\{a_1\cdots a_{k}\in\{0,1\}^k:\sum_{i=1}^k a_i\varphi^{k+2-i}\leq x_H\}\\
    &= \frac{1}{2^k}\sum_{i=0}^H R(i)\\
    &=\frac{1}{2^k}A(H)
\end{align*}
But $F_{k-1}\sim \frac{\varphi^{k-1}}{\sqrt{5}}$ and so
\[
\frac{1}{2^k}=\left(\frac{1}{\varphi^k}\right)^{\frac{\log 2}{\log\varphi}}\sim \left(\frac{1}{\varphi\sqrt{5}F_{k-1}}\right)^{\frac{\log 2}{\log\varphi}}
=\left(\frac{1}{H}.\frac{\gamma}{\varphi\sqrt{5}}\right)^{\frac{\log 2}{\log\varphi}}.\]
Also note that \[
\frac{x_H}{\varphi^{k+2}}=\frac{F_{k-1}\gamma}{\varphi^{k+2}}\sim\frac{\gamma}{\varphi^3}.
\]
Combining all of the above gives
\[
\dfrac{A(H)}{H^{\frac{\log 2}{\log \varphi}}}\sim \dfrac{G_{\varphi}
(\frac{\gamma}{\varphi^3})}{\gamma^{\frac{\log 2}{\log\varphi}}}(\varphi\sqrt{5})^{\frac{\log 2}{\log \varphi}}.
\]
Most of this argument could have been given with explicit error bounds, but in the final line we used that $G_{\varphi}(\frac{\gamma}{\varphi^3})$ and $G_{\varphi}(\frac{\gamma}{\varphi^3}+\frac{1}{\varphi^{k+1}})$ are close for large $k$, this follows from the uniform continuity of $G_{\varphi}$ but one would have to work quite hard to get explicit bounds.

As $H$ and hence $k$ grow, $\gamma$ stays approximately in the interval $[1,\varphi]$. Thus we see the $\log$ periodic pattern in Figure 3 is given by $c\frac{G_{\varphi}(x)}{x^{\frac{\log 2}{\log \varphi}}}$ for $x$ in the range $[\frac{1}{\varphi^3},\frac{1}{\varphi^2}]$ and $c$ an explicit constant.

\section*{Acknowedgements}
Many thanks to Sam Chow and Peej Ingarfeld for their useful comments on a previous draft of this article.
 
\bibliographystyle{abbrv} 
\bibliography{Fibonacci}

\begin{thebibliography}{10}

\bibitem{AFKP}
S.~Akiyama, D.-J. Feng, T.~Kempton, and T.~Persson.
\newblock On the {H}ausdorff dimension of {B}ernoulli convolutions.
\newblock {\em Int. Math. Res. Not. IMRN}, (19):6569--6595, 2020.

\bibitem{AlexanderZagier}
J.~C. Alexander and D.~Zagier.
\newblock The entropy of a certain infinitely convolved {B}ernoulli measure.
\newblock {\em J. London Math. Soc. (2)}, 44(1):121--134, 1991.

\bibitem{Ardila}
F.~Ardila.
\newblock The coefficients of a {F}ibonacci power series.
\newblock {\em Fibonacci Quart.}, 42(3):202--204, 2004.

\bibitem{BK1}
A.~Batsis and T.~Kempton.
\newblock Measures on the spectra of algebraic integers, 2021.

\bibitem{BK2}
A.~Batsis and T.~Kempton.
\newblock Towards absolutely continuous {B}ernoulli convolutions, 2021.

\bibitem{Berstel}
J.~Berstel.
\newblock An exercise on {F}ibonacci representations.
\newblock {\em Theor. Inform. Appl.}, 35(6):491--498, 2001.

\bibitem{Carlitz}
L.~Carlitz.
\newblock Fibonacci representations.
\newblock {\em Fibonacci Quart.}, 6(4):193--220, 1968.

\bibitem{ChowJones}
S.~Chow and O.~Jones.
\newblock On the variance of the fibonacci partition function, 2023.

\bibitem{ChowSlattery}
S.~Chow and T.~Slattery.
\newblock On {F}ibonacci partitions.
\newblock {\em J. Number Theory}, 225:310--326, 2021.

\bibitem{EdsonZamboni}
M.~Edson and L.~Q. Zamboni.
\newblock On representations of positive integers in the fibonacci base.
\newblock {\em Theoretical Computer Science}, 326(1):241--260, 2004.

\bibitem{HoggattBasin}
V.~E. Hoggatt and S.~L. Basin.
\newblock Representations by complete sequences — {P}art {I} ({F}ibonacci).
\newblock {\em Fibonacci Quart.}, 1:1--14, 1963.

\bibitem{Klarner}
D.~A. Klarner.
\newblock Partitions of {$N$} into distinct {F}ibonacci numbers.
\newblock {\em Fibonacci Quart.}, 6(4):235--244, 1968.

\bibitem{Rauzy}
G.~Rauzy.
\newblock Nombres alg\'{e}briques et substitutions.
\newblock {\em Bull. Soc. Math. France}, 110(2):147--178, 1982.

\bibitem{Weinstein}
F.~V. Weinstein.
\newblock Notes on {F}ibonacci partitions.
\newblock {\em Exp. Math.}, 25(4):482--500, 2016.

\bibitem{Zeckendorf}
E.~Zeckendorf.
\newblock Repr\'{e}sentation des nombres naturels par une somme de nombres de
  {F}ibonacci ou de nombres de {L}ucas.
\newblock {\em Bull. Soc. Roy. Sci. Li\`ege}, 41:179--182, 1972.

\bibitem{Zhou}
N.~H. Zhou.
\newblock On the representation functions of certain numeration systems.
\newblock 2023.

\end{thebibliography}
\end{document}